\documentclass[11pt,a4]{amsart}
\usepackage{amsmath,amssymb}
\usepackage{enumerate}
\usepackage{amsmath,amssymb}
\usepackage{graphics}

\newtheorem{thm}{Theorem}[section]
\newtheorem{prop}[thm]{Proposition}

\newtheorem{lemma}[thm]{Lemma}

\newtheorem{remark}[thm]{Remark}

\newcommand{\R}{\mathbb{R}}

\newcommand{\PROB}{{\mathbb{P}}}
\newcommand{\E}{{\mathbb{E}}}

\renewcommand{\eqref}[1]{(\ref{#1})}

\renewcommand{\P}{\mathbb{P}}

\renewcommand{\R}{\mathbb{R}}

\begin{document}
\bibliographystyle{plain}
\title{Boundary crossing identities for Brownian motion
 and some
 nonlinear
 ode's}
\vspace{.5cm}
\author{L. Alili}
\address{Department of Statistics, The University of Warwick,
Coventry CV4 7AL, United Kingdoms}
\email{l.alili@Warwick.ac.uk}


\author{P. Patie}
\address{School of Operations Research and Information Engineering, Cornell University, Ithaca, NY 14853.}
\email{pp396@cornell.edu}

\thanks{This work was partially supported by the Actions de Recherche Concert\'ees (ARC) IAPAS, a fund of the Communaut\'ee fran\c{c}aise de Belgique. The first author is greatly indebted to the Agence Nationale de la Recherche  for  the research grant ANR-09-Blan-0084-01. The second author would like to thank P. Lescot for several discussions on the Lie group method.}
\maketitle

\begin{abstract}
 We start by introducing a nonlinear  involution operator which maps the space of solutions of Sturm-Liouville equations into the space of solutions of  the associated equations which turn out to be nonlinear ordinary differential equations.
 We study some algebraic and analytical properties of this involution operator as well as some properties of  a two-parameter family of operators
describing the set of solutions of
Sturm-Liouville equations. Next, we show how a specific
composition of these mappings allows to connect, by  means of a simple analytical expression,  the law of the
first passage time of a Brownian motion over a curve to a
two-parameter  family of curves. We offer three different proofs of this fact  which may be of independent interests. In particular, one is  based on the construction of parametric time-space harmonic transforms of the law of some Gauss-Markov processes. Another one, which is of algebraic nature, relies on the Lie group symmetry  methods applied to the  heat equation and reveals that our two-parameter transformation is the  unique non-trivial one.
\end{abstract}

\vspace{5mm}

\section[Introduction and  main results]{Introduction and  main results}
\label{introduction-preliminaries} Let $B=(B_t)_{t\geq 0}$ be a
standard Brownian motion, starting at $0$, defined on a filtered probability space $(\Omega,(\mathcal{F})_{t\geq0},\mathcal{F},\P)$.   We are concerned with the distribution of the stopping time
  $$T^{f} = \inf \left\{t>0; \:B_t =
f(t) \right\}$$
where $f\in C([0,\infty), \mathbb{R})$ satisfies $f(0) \neq 0$ and
 the usual convention $\inf\{\emptyset \}=\infty$ applies. $C(I,J)$, for  some subintervals $I$ and  $J \subseteq \mathbb{R}$, stands for the space of continuous functions from $I$ into $J$. This boundary crossing problem has been  intensively studied for over a century and traces back to Bachelier's thesis \cite{Bachelier}. Although different methodologies have been suggested for solving it in some special instances, the finding of an explicit expression for the law of $T^f$ for a general curve remains an open problem.  We refer to \cite{Alili-patie-JTP-09} for a survey of these techniques. The  purpose of this paper is to describe a simple and explicit analytical expression relating the distributions of the first passage time of the Brownian motion to some two-parameter family of curves, extending the result obtained in  \cite{Alili-patie-JTP-09}. Besides, we shall provide three completely different proofs among which two may be of independent interests.   One of the proofs relies on the study of some Gauss-Markov processes and provides a procedure to produce Doob's $h$-transform of the law of these processes whenever the covariance functions are obtained as the image of   a two-parameter transformation relating solutions of a Sturm-Liouville equation. This approach motivated us to introduce and study two families of nonlinear transformations which turn out to be useful for describing the set of solutions of some strongly nonlinear second order differential equations in terms of solutions of the associated Sturm-Liouville equations.   Another proof, which is of algebraic nature, is based on the study of the Lie group symmetry of the heat equation, showing in particular that our main identity is the only non trivial one which is attainable.

Before stating our main results, we   introduce  some notations which will be used throughout the paper.  First, let the nonlinear operator $\tau$ be defined on the space of functions whose reciprocals are square integrable in some (possibly infinite) interval of $\R^+$  by
\[ \tau f(.)=\int_0^{.}f^{-2}(s) ds\]
and set
\begin{eqnarray*}
{A(a,b)}&=& \left\{ \pm f\in C\left([0,a), \mathbb{R}^+\right);
\: \tau f
(a) = b \right\}
\end{eqnarray*}
\noindent where $a$ and $b$ are positive reals.  Observe that if we denote by $A_{\infty}$ the set of continuous functions which are of constant sign on some non-empty interval with $0$ as left endpoint, then we have the following decomposition
\[
A_{\infty}= \bigcup_{ a
> 0} \bigcup_{
b>0}A(a,b).\]
We also introduce the family $(\Pi^{\alpha,\beta})_{(\alpha, \beta)
\in \mathbb{R}^*\times \mathbb R} $ of nonlinear operators  acting on  $A_{\infty}$  which are defined, for each  fixed couple of reals $\alpha$ and $\beta$, by
\begin{eqnarray} \label{eq:transf-pi}
\Pi^{\alpha,\beta} f = f\left(\alpha + \beta \tau f \right).
 \end{eqnarray}
Next, let $\varrho$ be the inversion operator acting on the space of continuous monotone functions, i.e. $\varrho f \circ f (t) =t$ where $\circ$ denotes the composition of functions.  We use the same symbol to define the composition of operators and, for example, by $\varrho \circ \tau f$ we mean the image by $\varrho$ of $\tau f$. Now, using
the nonlinear operator
\begin{eqnarray*}
  \Sigma  f  =\frac{1}{f \left(\varrho \circ \tau f\right)},
\end{eqnarray*}
  we construct the family $(S^{\alpha,\beta})_{\alpha\in \R^*, \beta \in \mathbb{R}}$  of operators as follows
 \begin{equation}\label{defin-S}
S^{\alpha, \beta} f=\Sigma\circ\Pi^{\alpha,-\beta}\circ
 \Sigma f.
 \end{equation}
Finally, for all $a>0$, we set
\begin{equation*}
a_{{\alpha,\beta}} =
\left\{
\begin{array}{ll}
\frac{a}{\alpha(\alpha-\beta a)}  \quad & \textrm{ if } \alpha(\alpha-\beta a)>0,\\
+\infty &  \textrm{ otherwise,}
\end{array}
\right.
\end{equation*}
and, for $f\in A_{\infty}$, we write
\begin{equation*}
a_{\alpha, \beta}^{f} =
\left\{
\begin{array}{ll}
  \quad a & \textrm{ if } \alpha(\alpha-\beta a)>0,\\
\varrho \circ \tau f\left(\frac{\alpha}{\beta}\right) &  \textrm{ otherwise.}
\end{array}
\right.
\end{equation*}
We note that, as $a\rightarrow \infty$, we have
\begin{equation*}
a_{{\alpha,\beta}} \rightarrow \zeta_{{\alpha,\beta}} =
\left\{
\begin{array}{ll}
-\frac{1}{\alpha\beta }  \quad & \textrm{ if } \alpha\beta <0,\\
+\infty &  \textrm{ otherwise,}
\end{array}
\right.
\end{equation*}
and the inequality ${a}_{\alpha, \beta}\leq \zeta_{\alpha, \beta}$ holds. Now, we are ready to state the first main result of this paper whose proof is postponed to Section \ref{section-gou-ou}.

\begin{thm} \label{thm1}
\begin{enumerate}
\item For $\alpha\neq 0$ and $\beta$ reals, the mapping $S^{\alpha, \beta}:A_{\infty} \rightarrow A_{\infty}$ is a linear operator admitting the simple representation
\begin{eqnarray} \label{elementary-mappings}
S^{\alpha, \beta}f(t)=\left(\frac{1+ \alpha \beta t
}{\alpha}\right)f\left(\frac{\alpha^2 t}{1+ \alpha \beta t}\right).
\end{eqnarray}
Furthermore, if $ f \in A(a,b)$ then $S^{\alpha,\beta} f \in  A(a_{{\alpha,\beta}},b^{f}_{\alpha, \beta})$.
\item Let $\mu$ be a positive Radon measure  on $\R^{+}$. Then,
there exists a unique positive, increasing, concave and differentiable function $\mathfrak{f}$  with $\mathfrak{f}(0)=1$, which satisfies the following nonlinear differential equation
 \begin{eqnarray} \label{eq:dnleq} f^3 f''
=-\mu\left(\tau {f}\right) \end{eqnarray}
on $\mathbb{R}^+$,
where $f''$ is the second derivative of $f$ considered in
 the sense of distributions. Furthermore, $\{S^{\alpha,\beta} \mathfrak{f};\: \alpha>0,\beta \geq0\}$ is the set of positive solutions of (\ref{eq:dnleq}).
\end{enumerate}
\end{thm}

We carry on by providing an example illustrating Theorem \ref{thm1}. To this end, let us consider the positive measure $\mu_{a,b}$,  for some fixed $a>0$ and $b>0$, which is specified on $\R^+$ by
\begin{equation*}
\mu_{a,b}(dt) = \frac{a}{\left(1+bt\right)^{2}}dt.
\end{equation*}
We easily check that the function $\mathfrak{f}_{\gamma}$  given by
\begin{equation} \label{eq:fg}
\mathfrak{f}_{\gamma}(t)=(\kappa t +1)^{\gamma}
\end{equation}
where  $\kappa= \sqrt{4a+b^2} \quad \hbox{and}\quad \gamma =\frac{1}{2}\left(1-\frac{b}{\kappa}\right) \in (0,1)$
  satisfies the requirements of item {\em 2.~}of Theorem \ref{thm1}. That is $\mathfrak{f}_{\gamma}$ is positive, concave and increasing and it solves the nonlinear differential equation \eqref{eq:dnleq} with $\mu=\mu_{a,b}$. Moreover, for any $\alpha,\beta>0$, the function
\begin{eqnarray*}
S^{\alpha,\beta} \mathfrak{f}_{\gamma}(t)&=& \frac{((\kappa\alpha^2+\alpha\beta  )t+1)^{\gamma}}{\alpha(1+\alpha \beta t)^{\gamma-1}}
\end{eqnarray*}
is also a positive solution of \eqref{eq:dnleq}.

We proceed by pointing out that the mapping $ S^{\alpha, \beta}$
can also be defined on the space of probability measures. For instance, in the absolutely
continuous case, we associate to $\mu(dt) = h(t)dt$ the image $S^{\alpha, \beta}(
\mu)(dt) = S^{\alpha, \beta}h(t) dt$. We are now ready to  state our second main result which relates the distributions of the family of stopping times $(T^{S^{\alpha, \beta}f})_{\alpha\in\R^*,\beta\in \R}$.
\begin{thm} \label{thm2}
Let $f \in C\left([0,\infty), \mathbb{R}\right)$  be such that $f(0) \neq 0$ and $\alpha\neq 0$, $\beta$ two fixed reals. Then, for any  $t<\zeta_{\alpha, \beta}$, we have the relationship
\begin{equation}\label{eq:fptr}
\mathbb{P}\left(T^{S^{\alpha, \beta} f}\in
dt\right)=\alpha^3 \left(1 + \alpha\beta
t\right)^{-\frac{5}{2}}e^{-\frac{\alpha \beta
}{2(1+ \alpha\beta
t)}(S^{\alpha, \beta} f(t))^2}S^{\alpha,\beta}\left(\mathbb{P}( T^{f} \in dt)\right).
\end{equation}
\end{thm}
We provide in Section \ref{Doob} below three proofs of this Theorem.  We now use this result to compute the distribution of the stopping time
\[ T^{S^{\alpha,\beta} \mathfrak{f}_{2}}=\inf\left\{ 0<t<\zeta_{\alpha, \beta};\:
B_t=\frac{((\kappa\alpha^2+\alpha\beta )t+1)^{2}}{\alpha(1+\alpha \beta t)}\right\}
\]
where $\alpha\neq 0$, $\beta$ are reals and $\mathfrak{f}_{2}$ is defined in \eqref{eq:fg}. We recall that  Salminen \cite{Sal}, see also
Groeneboom \cite{Groeneboom}, has derived the probability density function of the distribution of
$T^{f_2}$ where  $f_2(t)=1+\kappa^2t^2$.  Writing $p^{f_2}(t)dt=\P(T^{f_2} \in dt)$, he found the latter to be
\begin{equation*}
p^{f_2}(t) = 2(\kappa^2 c)^{2} e^{-\frac{2}{3}\kappa^4
t^3}
\sum_{k=0}^{\infty}\frac{Ai\left(z_k+2c\kappa^2\right)}{Ai'(z_k)}e^{\frac{z_k}{c}
t}
\end{equation*}
for $t>0$, where $(z_k)_{k \geq 0}$ is the decreasing sequence of negative
zeros of the Airy function $Ai$ and
$c=(2\kappa^4)^{-\frac{1}{3}}$. Now,  by means of  the
Cameron-Martin formula, we obtain, with $\mathfrak{f}_{2}(t)=(1+\kappa t)^2=f_2(t)+2\kappa t$,
\begin{equation*}
p^{\mathfrak{f}_2}(t) = 2(\kappa^2 c e^{-\kappa})^{2} h_{\kappa}\left(t\right)\sum_{k=0}^{\infty}\frac{Ai\left(z_k+2c\kappa^2\right)}{Ai'(z_k)}e^{\frac{z_k}{c} t}
\end{equation*}
for $t>0$, where    $-\log h_{\kappa}(t)=2\kappa^2 t
\left(1+\frac{1}{3}\kappa^2 t^2+\kappa t\right)
.$ Finally, by applying Theorem \ref{thm2}, we obtain
\begin{equation*}
p^{S^{\alpha,\beta} \mathfrak{f}_{2}}(t) = \frac{2\alpha^{2}(\kappa^2 ce^{-\kappa})^{2}}{(1+\alpha \beta t)^{3/2}}e^{-\frac{ \beta((\kappa\alpha^2+\alpha\beta )t+1)^{4}
}{2\alpha(1+ \alpha\beta
t)^3}}h_{\kappa}\left(\frac{\alpha^2t}{1+\alpha \beta t}\right)\sum_{k=0}^{\infty}\frac{Ai\left(z_k+2c\kappa^2\right)}{Ai'(z_k)}e^{
\frac{z_k\alpha^2 t}{c(1+\alpha \beta t)}}
\end{equation*}
for $ t<\zeta_{\alpha, \beta}$.
Note that if we take $\beta = -\alpha \kappa$ then we recover the example treated in \cite{Alili-patie-JTP-09}.

\section{Proof of  Theorem \ref{thm1}} \label{section-gou-ou}
We start by listing  some basic algebraic and analytic properties of the operator $\Sigma$.
\begin{prop} \label{prop:sigma}
For any $a,b>0$, $\lambda \in \R$ and $f\in A_{\infty}$, we have the following assertions.

\begin{enumerate}
\item  $\tau = \varrho \circ \tau \circ \Sigma$.
\item  $ \Sigma
\left( A(a,b) \right)= A({b},{ a})$.
\item $ \Sigma$ is an involution operator, that is $ \Sigma \circ \Sigma f = f$.
\item  $\lambda \Sigma \lambda f = (\Sigma f)_{\lambda}$ with $f_{\lambda}(t)=f(\lambda^2t)$. In particular, $ \Sigma  (-f) = -\Sigma f$.
\item If $f$ is increasing (resp.~ differentiable and convex) then $\Sigma f$ is decreasing (resp.~ differentiable and concave).
\end{enumerate}
\end{prop}
\begin{proof} We obtain the
first assertion by observing that
$\tau \circ \Sigma f (t)= \int_0^{t}(f(\varrho \circ \tau f(s))^2ds=\varrho \circ \tau f(t)$ and using the fact that $\varrho$ is an involution. Next, let $f \in A(a,b)$ and note that  $\tau f$ is an homeomorphism from
$[0,a)$ into $[0,b)$. Thus, the mapping  $t\mapsto \Sigma f(t)$ is plainly continuous and positive on $[0,b)$.  Hence, observing that $\varrho \circ \tau f (b)=a$, the second item follows from the previous one.
We deduce from the second statement that  $ \Sigma \circ \Sigma
f = \Sigma \left( 1/f\left( \varrho \circ \tau f \right)\right) =f\left( \varrho \circ \tau f\right)
  = f$ which gives the third statement.  The fourth statement is an easy consequence of the identity $\varrho \circ \tau \lambda f (t) = \varrho \circ \tau f (\lambda ^2 t)$. The first claim of the  last assertion is straightforward.  Finally, noting that $(\Sigma f)'(t)=-f'\left(\varrho\circ \tau f(t)\right)$, the last item is obtained by using the fact that the mapping $t \mapsto \varrho\circ \tau f(t)$ is increasing.
\end{proof}
In what follows, we describe some interesting properties satisfied by the family of linear operators $(\Pi^{\alpha,\beta})_{\alpha\in\R^*,\beta\in \R}$ defined in \eqref{eq:transf-pi}. Before doing that, we recall how these operators  are related to a class of ordinary second order differential equations. More precisely, recalling that  $\mu$ denotes a positive Radon measure  on $\R^+$, we consider  the following  Sturm-Liouville equation
\begin{equation}\label{Sturm-Liouville-equation}
\phi''=\mu \phi
\end{equation}
where $\phi''$ is defined in the sense of distributions. Clearly, if $\phi$ is a solution to (\ref{Sturm-Liouville-equation}) then
so is $\Pi^{0,1}\phi$. Actually, the set of solutions to equation (\ref{Sturm-Liouville-equation}) is the vectorial
space $\{\Pi^{\alpha,\beta}\phi=\alpha \phi+\beta \Pi^{0,1}\phi;\alpha,\beta \in \R\}$. Observe, that all positive solutions are
convex and  described by the set $\{\Pi^{\alpha,\beta}\varphi=\alpha \varphi+\beta \Pi^{0,1}\varphi;\alpha >0,\beta \geq0\}$ where $\varphi$ is the unique
positive decreasing solution satisfying
$\varphi(0)=1$.  Moreover, $\varphi$ satisfies $\lim_{t\rightarrow
\infty}\varphi(t)\in[0,1]$ and the strict inequality
$\varphi(\infty)<1$ except in the trivial case $\mu\equiv0$ which we exclude. We point out that $\varphi$ is also differentiable on the support of $\mu$. Moreover,
under the condition $\int (1+s)\mu(ds) <\infty$ we know that
$\lim_{t\rightarrow
\infty}\varphi(t) >0$. We refer to \cite[Appendix \S 8]{Revuz-Yor-99} for a detailed account on these facts.
We are now ready to state the following result where the study is restricted to $\alpha\geq 0$ since the other case can be recovered by using  the identity $\Pi^{-\alpha,\beta} = -\Pi^{\alpha,-\beta}$.
\begin{prop} \label{prop:pi-prop}  Let $ (\alpha,\beta) \in \R^+ \times\R$, $(\alpha', \beta')\in[0,\infty)\times \R$ and $\phi\in A(a,b)$ for some positive reals $a$ and $b$. Then, we have the following assertions.
\begin{enumerate}
\item   $\Pi^{\alpha,\beta}=\alpha\Pi^{1,\beta/\alpha}$.
\item  $\Pi^{\alpha,\beta} \phi\in A(b_{\alpha, -\beta},a_{\alpha,-\beta}^{\phi})$.
\item $\Pi^{\alpha,\beta}  \circ \Pi^{\alpha',\beta'}  =
\Pi^{\alpha \alpha', \alpha\beta' + \beta/\alpha'}.$  In
particular, $\Pi^{\alpha,\beta}$  is the inverse operator of $\Pi^{1/\alpha,-\beta} $ and $(\Pi^{1,\beta})_{\beta \geq 0}$ is a semigroup.
\end{enumerate}
\end{prop}
\begin{proof}
The first item is obvious. Next, recalling that
$\Pi^{\alpha,\beta} \phi = \phi (\alpha + \beta \tau {\phi})$, we obtain
\begin{equation} \label{eq:i1}
\tau \circ {\Pi^{\alpha,\beta} \phi}(.) =\frac{1}{\alpha}\frac{\tau {\phi}(.)}{\alpha + \beta \tau {\phi}(.)}.
\end{equation}
Now, if $\beta>0$ then  $\Pi^{\alpha,\beta} \phi $ is continuous and positive on $[0,a)$ with
 $\tau \circ {\Pi^{\alpha,\beta}
\phi}(a)=\frac{1}{\alpha}\frac{b}{\alpha + \beta b}.$  If $\beta <0$ then
$\Pi^{\alpha,\beta} \phi $ is continuous and positive on $\left[0,a \wedge
\varrho \circ \tau {\phi}\left(-\frac{\alpha}{\beta}\right)\right)$ and we have  $a<\varrho \circ \tau {\phi}\left(-\frac{\alpha}{\beta}\right)$ when
$\alpha + \beta b>0$. Thus, the second statement follows from $\tau\circ {\Pi^{\alpha,\beta}\phi}\left(\varrho \circ \tau {\phi}\left(-\frac{\alpha}{\beta}\right)\right)
= \infty$. Next, we readily deduce from \eqref{eq:i1} that
\begin{eqnarray*}
\Pi^{\alpha,\beta} \circ \Pi^{\alpha',\beta'} \phi &=& \Pi^{\alpha,\beta} \phi \left(\alpha'+\beta'\tau {\phi}\right)\\
&=& \phi \left( \alpha'+ \beta'\tau {\phi}\right)\left( \alpha+ \frac{\beta}{\alpha'} \frac{\tau{\phi}}{\alpha'+\beta'\tau \phi}\right)\\
&=& \phi \left(\alpha \alpha'+ \left(\alpha\beta'+\frac{\beta}{\alpha'}\right)\tau {\phi}\right)\\
&=&\Pi^{\alpha \alpha', \alpha\beta' + \beta/\alpha'}\phi
\end{eqnarray*}
which completes the proof of the Proposition.
\end{proof}
Now, we are  ready to study some properties of  the family of linear operators  $(S^{\alpha, \beta})_{\alpha \in \R^*,\beta\in \R}$ which is defined in  \eqref{defin-S}. In particular, the next result contains the claims of item {\em{1.~}}of Theorem \ref{thm1}.
 \begin{prop} \label{prop:s} Let $ (\alpha,\beta)$ and $(\alpha', \beta') \in \R^+ \times\R$. Then, the following assertions hold true.
 \begin{enumerate}
 \item If $ f \in A(a,b)$ then $S^{\alpha,\beta} f \in  A(a_{\alpha,\beta},b^{f}_{\alpha, \beta})$. Moreover, formula (\ref{elementary-mappings}) holds true.
\item  $
S^{\alpha,\beta}\circ
S^{\alpha',\beta'}=S^{\alpha\alpha',\alpha\beta'+\frac{\beta}{\alpha'}}$.
In particular, $(S^{1,\beta})_{\beta \geq 0}$ is a semigroup of linear operators.
\item Assuming that $f \in A(a,b)$ is concave and differentiable  then we have the following statements.
\begin{enumerate}
\item $S^{\alpha,\beta}f $ is also concave and  differentiable.
\item  If  $f(0)>0$ and
$f\left(\frac{1}{\beta_0}\right)=0$ for some $\beta_0>0$,  then, for any
$\frac{\beta}{\alpha}\geq \beta_0$, $S^{\alpha,\beta}f$ is  non-decreasing on $\R^+$.
 \end{enumerate}
  \end{enumerate}
 \end{prop}
 \begin{proof}
Item {\em{1.~}}and the first part of item {\em{2.}}~follow readily  from the definition of $S^{\alpha,\beta}$  and propositions \ref{prop:sigma} and \ref{prop:pi-prop}. Then, from Proposition \ref{prop:sigma}, we get that
\begin{eqnarray} \label{eq:ps}
\Pi^{\alpha,-\beta} \circ \Sigma f (t) = \frac{\alpha-\beta \varrho \circ \tau f}{f ( \varrho \circ \tau f )}(t).
\end{eqnarray}
Moreover, we see that
\begin{eqnarray*}
\tau \circ  {\Pi^{\alpha,-\beta} \circ \Sigma f} (t) =
\frac{1}{\alpha}\frac{\varrho \circ \tau f}{\alpha - \beta \varrho \circ \tau f}(t).
\end{eqnarray*}
Inverting yields
\begin{equation}\label{eq:vps}
\varrho\circ \tau \circ  \Pi^{\alpha,-\beta} \circ \Sigma f (t) = \tau f\left(\frac{\alpha^2 t}{1+\alpha \beta t}\right).
\end{equation}
Finally, combining  \eqref{eq:ps} and \eqref{eq:vps}, we can write
\begin{eqnarray*}
 S^{\alpha,\beta} f (t)
 &=& \Sigma  \frac{\alpha-\beta \varrho \circ \tau f}{f \left(\varrho \circ \tau f  \right)}(t)
\\
&=&  \frac{f  (\varrho \circ \tau f)}{\alpha-\beta \varrho \circ \tau f} \circ  \tau f \left(\frac{\alpha^2 t}{1+\alpha \beta t}\right)
\end{eqnarray*}
which is easily simplified to get \eqref{elementary-mappings}. It is clear from \eqref{elementary-mappings} that $S^{\alpha,\beta}$ is a linear operator. Next, since $\Sigma$ is an involution, item {\em{2.}}~follows from Proposition \ref{prop:pi-prop}
and
\begin{eqnarray*}
S^{\alpha,\beta}\circ S^{\alpha',\beta'} &=& \Sigma \circ  \Pi^{\alpha,-\beta}\circ \Pi^{\alpha',-\beta'}  \circ \Sigma\\
 &=& \Sigma \circ  \Pi^{\alpha \alpha', -\alpha\beta' - \beta/\alpha'}  \circ \Sigma\\
  &=& S^{\alpha \alpha', \alpha\beta' + \beta/\alpha'}.
 \end{eqnarray*}
Item {{\em 3.a.~}}follows readily from the definition of $(S^{\alpha, \beta})_{\alpha \in \R^+,\beta\in \R}$ combined with the propositions \ref{prop:sigma} and \ref{prop:pi-prop}.  Finally, we note that for $\frac{\beta}{\alpha} \geq \beta_0$,
$S^{\alpha,\beta}f$ is positive since $t\mapsto \alpha^2 t/(1+\alpha \beta t)$ is increasing on $\R^+$, which completes the proof by means of the concavity property.
\end{proof}

We are now ready to  complete the proof of Theorem \ref{thm1} by proving item {{\em 2.}} We  need to show that the image of \eqref{eq:dnleq} by $\Sigma$ is equation \eqref{Sturm-Liouville-equation}. To that end, assume that $f$ satisfies \eqref{eq:dnleq} and
let us write $\phi=\Sigma f $. We deduce from the relationship $\phi(\tau f(\cdot)) f(\cdot) =1$ that $\phi'(\tau f(\cdot)) = -f'(\cdot)$ and
$f''(\cdot) = - \phi''(\tau f(\cdot))/{f^{2}(\cdot)}$
in the sense of distributions. Using \eqref{eq:dnleq} we obtain $\phi''(\tau f(\cdot))=\mu\circ \tau f(\cdot) \phi\circ \tau f(\cdot)$. Thus, $\phi$ solves \eqref{Sturm-Liouville-equation}. Conversely, similar arguments show that if $\phi$ solves \eqref{Sturm-Liouville-equation} then $f=\Sigma \phi$ solves \eqref{eq:dnleq}. It follows from Proposition \ref{prop:sigma} that  the function  $\mathfrak{f}=\Sigma \varphi$, where $\varphi$ is defined  just before Proposition \ref{prop:pi-prop}, satisfies the required properties. We conclude that there exists a unique increasing, concave and differentiable function $\mathfrak{f}$ such that $\mathfrak{f}(0)=1$.

\begin{remark}  If $f \in C([0,\infty),\R^+)$ is a  solution to \eqref{eq:dnleq}  then it admits the representation  $f(t)=S^{1/f(0),\beta}\mathfrak{f}(t)$ where $\beta=f'(0)-\mathfrak{f}'(0)/f(0)$ and if $\beta>0$ then $
\lim_{t\rightarrow \infty} f(t)/t= \beta \mathfrak{f}(\alpha/\beta)>\beta$.
\end{remark}

\section{Proof of  Theorem \ref{thm2}}\label{Doob}
We actually derive three proofs of Theorem \ref{thm2}. The first one is  almost straightforward and hinges on a previous result obtained by the authors   in \cite{Alili-patie-JTP-09} whereas the second one reveals some interesting results concerning time-space harmonic transforms of the law of  Gauss-Markov processes and explains the connections with the analytical result stated   in Theorem \ref{thm1}. The last one relies on the Lie group techniques applied to the heat equation.  Before developing the  proofs,  we mention that the symmetry  of the Brownian motion implies the following identity in distribution
\begin{equation} \label{eq:idsa}
T^{S^{\alpha,\beta}f}\stackrel{d}{=}T^{S^{|\alpha|,\textrm{sgn}(\alpha)\beta}f}
\end{equation}
for any $(\alpha,\beta) \in \R^*\times \R$.  Hence, it is enough to consider the case $\alpha>0$.  For convenience, we set $f^{\alpha,\beta}=S^{\alpha,\beta}f$.
\subsection{The direct approach}
   We get from item {\em{2.}}~of Proposition \ref{prop:s} that  $S^{\alpha,\beta}=   S^{1,\alpha\beta} \circ S^{\alpha,0}$ which when combined with \cite[Theorem 1]{Alili-patie-JTP-09} gives our result. To be more precise, recall that, from the aforementioned reference,  we have
   \begin{equation*}\label{switching-identity}
\mathbb{P}\left(T^{f^{1,\beta}}\in dt\right)=(1+\beta
t)^{-5/2}e^{-\frac{\beta}{2}{\frac{({f^{1,\beta}}(t))^2}{1+\beta
  t}}}
S^{1,\beta}\left(\mathbb{P}\left( T^{f} \in dt\right)\right)
\end{equation*}
for all $t<\zeta_{1,\beta}$. Thus, by using  $f^{\alpha,\beta}=   S^{1,\alpha\beta} \circ S^{\alpha,0}f $, we can write \begin{eqnarray*}
  \mathbb{P}\left(T^{f^{\alpha,\beta}}\in dt\right) &=& (1+\alpha \beta t)^{-5/2} e^{-\frac{\alpha \beta}{2(1+\alpha \beta t)} (S^{1,\alpha \beta} \circ S^{\alpha,0} f(t))^2}S^{1,\alpha \beta}\left(\mathbb{P}\left(T^{S^{\alpha,0} f}\in dt\right)\right),
  \end{eqnarray*}
  for $t<\zeta_{\alpha, \beta}$.
    Next, using the scaling property of $B$, we obtain the equality in distribution $ T^{f}\stackrel{d}{=}\alpha^2T^{f^{\alpha,0}}$ from which  we easily deduce that
$  \mathbb{P}\left(T^{f^{\alpha,0} }\in dt\right)=\alpha^3 S^{\alpha,0}\left(\mathbb{P}\left(T^f\in dt\right)\right)$. Using the linearity  and again the composition properties  of $S^{\alpha,\beta}$, we  get
    \begin{eqnarray*}
  \mathbb{P}\left(T^{f^{\alpha,\beta}}\in dt\right) &=& \alpha^3 (1+\alpha \beta t)^{-5/2} e^{-\frac{\alpha \beta}{2(1+\alpha \beta t)} (f^{\alpha, \beta}(t))^2}S^{\alpha, \beta}\left(\mathbb{P}\left(T^f\in dt\right)\right)
  \end{eqnarray*}
  which combined with formula \eqref{eq:idsa} completes the proof of Theorem \ref{thm2}.
\subsection{The proof via Gauss-Markov processes}
  For the second approach,  we take  $\phi
\in A_{b}^{a} \cap AC([0,b))$, where $AC([0,b))$ is the space of absolutely continuous functions  on $[0,b)$, and consider the associated Gauss-Markov process of
Ornstein-Uhlenbeck type with parameter $\phi$. More specifically, we denote by  $\P^{\phi}=(\P^{\phi}_x)_{x\in \R}$ the family of probability measures  of the process $X=(X_t)_{0\leq t<b}$ which is defined to be the unique strong solution to the stochastic differential equation
\begin{equation*} \label{eq:def1}
dX_t = \frac{\phi'(t)}{\phi(t)}X_t dt + dB_t,
 \quad X_0=x,
\end{equation*}
for $0\leq t<b$. Clearly, $X$ is the Gaussian process  given, for each fixed $0\leq t<b$, by
 \begin{equation*} \label{eq:def2}
X_t = \phi(t) \left( x +  \int_0^t
\phi^{-1}(s)  \: dB_s \right),
\end{equation*}
which has mean and covariance function
\begin{equation*}
m(t)= x\phi(t), \quad
v(s,t)=\phi(t\vee s)\Pi^{0,1}\phi(s\wedge t), \quad 0\leq s, t\leq b,
\end{equation*}
respectively. To simplify, we assume throughout that $\phi(0)=1$. Note that if we take $\phi(t)=e^{-\lambda t}$, for some $\lambda
>0$, then $X$ is the classical Ornstein-Uhlenbeck process. Moreover,
if $X_0$ is a centered and normally distributed random variable,  with variance $1/2
\lambda$, which is independent of $B$, then $X$ is the unique  Gauss-Markov process which is stationary  see e.g \cite[Excercise (1.13), p.86]{Revuz-Yor-99}. Our motivation for introducing this process stems from the following simple connection between two types of boundary crossing problems.
\begin{lemma} \label{Proposition-connecting-times}
Let, for any $y \in \R$, $T_y=\inf\{0<t<b;\: \phi(t)   \int_0^t
\phi^{-1}(s)  \: dB_s =y\}$. Then, for any $f \in A(a,b)$,  writing $\phi=\Sigma f$ and $T=T_1$,   the identity
\begin{equation}\label{eq:doob-representation}
T^f =\tau \phi \left(T\right)
\end{equation}
holds almost surely.
In particular, $\PROB_0^{\phi}\left(T<b
\right)=\PROB\left(T^{f} <a\right)$.
\end{lemma}
\begin{proof} By means of Dumbis, Dubins-Schwarz theorem, see e.g.~\cite[Theorem V.1.6]{Revuz-Yor-99}, there exists a standard Brownian motion $(W_t)_{0\leq t<b}$  defined on $(\Omega,\mathcal{F},\P)$ such that we have a.s.
\begin{eqnarray*}
T& = & \inf\left\{t > 0; \: \phi \left(t\right)W_{\tau\phi(t)} =  1  \right\}    \\
& = & \varrho \circ \tau \phi (T^{\Sigma \phi})\\
&=& \tau f (T^{f})
\end{eqnarray*}
where  we used for the last identity item {\em 1.~}of Proposition \ref{prop:sigma}. The proof is now easy to complete.
\end{proof}
We mention that
relation \eqref{eq:doob-representation} was  used by Breiman \cite{Breiman-67} for relating the first crossing time of a Brownian motion over the square root boundary to the first passage time to a fixed level by the classical stationary
Ornstein-Uhlenbeck process.  Next, we need to introduce the notation
\begin{equation}
H_t(x)
 =\left(\frac{\alpha \phi(t)}{\Pi^{\alpha,\beta}\phi(t)}\right)^{\frac{1}{2}}e^{\frac{\beta}{2}
\frac{x^2}{\phi(t)\Pi^{\alpha,\beta}\phi(t)}}.
\end{equation}
Our aim now is to show that the parametric families of distributions $(\P^{\Pi^{\alpha,\beta}\phi})_{(\alpha,\beta)\in  \R^* \times \R}$ of Gauss-Markov processes are related by some simple space-time harmonic transforms.
\begin{lemma} \label{thm-gmou} For $(\alpha,\beta) \in \R^* \times \R$ and $\phi$ as above, the process $\left(H_t(X_t)\right)_{ 0 \leq t<a^{\phi}_{\alpha,-\beta}} $
is a $\PROB^{\phi}$-martingale. Furthermore,  the absolute-continuity relationship
\begin{equation}\label{absolute-continuity-ous}
   d\PROB^{\Pi^{\alpha,\beta}\phi}_{x|\mathcal{F}_t}=\frac{H_t(X_t)}{H_0(x)}d{\PROB}^{\phi}_{x|\mathcal{F}_t}
\end{equation}
holds for all $x \in \R$ and $
t < a^{\phi}_{\alpha,-\beta}$.
Consequently, for any reals $x$
and  $y$, we have
\begin{equation}
\PROB^{\Pi^{\alpha,\beta}\phi}_x\left(T_y\in dt\right)=
\frac{H_t(y)}{H_0(x)} \: \PROB^{\phi}_x\left(T_y\in
dt\right), \quad t<a^{\phi}_{\alpha,-\beta}.
\end{equation}
\end{lemma}
\begin{proof}  First,  the It\^o  formula   yields
\[ \frac{\beta}{2 }\frac{B_t^2}{\alpha+\beta t} =\beta \int_0^t \frac{B_s}{\alpha+\beta s}dB_s-\frac{\beta^2}{2} \int_0^t \frac{B^2_s}{(\alpha+\beta s)^2}ds-\frac{1}{2}\log(\alpha+\beta t). \]
Thus, in the special case $\phi\equiv 1$,  the process $\left(H_t(X_t)\right)_{0 \leq t<a^{\phi}_{\alpha,-\beta}}$ is a $\P$-local martingale.  Moreover, from the well-known identity $\E\left[e^{-\frac{\lambda}{2} B_t^2}\right]=(1+\lambda t)^{-1/2},\:\lambda >-1/t$, see e.g. \cite[p.441]{Revuz-Yor-99}, we deduce that, for all $t <a^{\phi}_{\alpha,-\beta}$, we have
$\E[H_t(B_t)]=1$. Hence it is a true martingale. The $\mathbb P^{\phi}$-martingale property of $\left(H_t(X_t)\right)_{ 0 \leq t<a^{\phi}_{\alpha,-\beta}}$ follows from the fact it has
 the same distribution as
the process $$\left(H_{\tau {\phi}(t)}(X_{\tau {\phi}(t)})\right)_{ 0 \leq t<a^{\phi}_{\alpha,-\beta}}$$ under $\P$.
Next, since we have on the one hand
\[ d\left<\beta \int_0^{\tau\phi(.)} \frac{B_s}{\alpha+\beta s}dB_s,X_.\right>_t=\beta \frac{X_t}{\Pi^{\alpha,\beta}\phi(t)}dt\]
and on the other hand
\[\frac{(\Pi^{\alpha,\beta}\phi)'(t)}{\Pi^{\alpha,\beta}\phi(t)} = \frac{\phi'(t)}{\phi(t)}+\frac{\beta}{\Pi^{\alpha,\beta}\phi(t)},\]
we deduce the absolute continuity relationship by an application of  Girsanov's theorem. Next, on the event $\{T_y
\leq t\} \in \mathcal{F}_{t\wedge T_y}$, we have $H_{t\wedge T_y}(X_{t\wedge T_y})= H_{T_y}(y)$. Now, Doob's optional
stopping theorem implies that
\begin{eqnarray*}
   \PROB^{\Pi^{\alpha,\beta}\phi}_x(T_y \leq t )&=&\E^{\phi}_x\left[{\mathbb{I}}_{\{T_y \leq t
   \}}\frac{H_t(X_t)}{{H_0(x)}}\right]\\
   &=&\E^{\phi}_x\left[{\mathbb{I}}_{\{T_y \leq t
   \}}\E^{\phi}_x\left[\frac{H_t(X_t)}{H_0(x)}\big|\mathcal{F}_{t\wedge T_y}\right]\right]\\
   &=& \E^{\phi}_x\left[\frac{H_{T_y}(y)}{H_0(x)}{\mathbb{I}}_{\{T_y \leq t
   \}}\right].
\end{eqnarray*}
Our claim follows then by differentiation.
\end{proof}
Now, we are ready to complete a version of the second proof of Theorem \ref{thm2}. For the sake of clarity, we assume that $f$ is  continuously differentiable and thus according to Strassen \cite{Strassen-67}, the law of $T^f$ is absolutely continuous with a continuous density which we denote by $p^f$.
Next, let $\phi=\Sigma f$ and thus, by definition, $f^{\alpha,\beta} = \Sigma \circ \Pi^{\alpha,-\beta} \phi$. Since $\Sigma$ is an involution, we have from  Lemma \ref{Proposition-connecting-times} that  $T^{f^{\alpha,\beta}}=\tau \circ \Pi^{\alpha,-\beta} \phi \left(T\right) a.s.$  Using the fact that
\begin{equation} \label{eq:it}
\varrho \circ \tau \circ \Pi^{\alpha,-\beta} \phi = \varrho \circ \tau \circ \Sigma f^{\alpha,\beta} = \tau f^{\alpha,\beta},
\end{equation}
we get,  for any $t<\zeta_{\alpha, \beta}$,
\begin{eqnarray*}
p^{{f^{\alpha,\beta}}}\left(t\right)dt &=&f^{\alpha,\beta}(t)^{-2}
\P_0^{\Pi^{\alpha,-\beta} \phi}\left(T\in  \tau f^{\alpha,\beta} (t) \right)\\
&=&f^{\alpha,\beta}(t)^{-2}
\P_0^{\Pi^{\alpha,-\beta} \phi}\left(T\in  \tau f \left(\frac{\alpha^2 t}{1+\alpha \beta t}\right) \right)
\end{eqnarray*}
where we used the identity $\tau f^{\alpha,\beta}  (t)=\tau f \left(\frac{\alpha^2 t}{1+\alpha \beta t}\right)$ which follows readily by a change of variable.
Then, Proposition  \ref{thm-gmou} combined with the identities
\begin{eqnarray*}
\left(\Pi^{\alpha,-\beta}\phi(\varrho \circ \tau \circ \Pi^{\alpha,-\beta} \phi)\right)^{-1} &=&  \Sigma \circ \Pi^{\alpha,-\beta}\phi = f^{\alpha,\beta}, \\
 \left(\phi\left(\tau f \left(\frac{\alpha^2 t}{1+\alpha \beta t}\right)\right)\right)^{-1} &=&  \left(\phi\left(\varrho \circ \tau \phi \left(\frac{\alpha^2 t}{1+\alpha \beta t}\right)\right)\right)^{-1}=f\left(\frac{\alpha^2 t}{1+\alpha \beta t}\right),
 \end{eqnarray*}
  yields, for any $t<\zeta_{\alpha,\beta}$,
  \begin{eqnarray*}
\P_0^{\Pi^{\alpha,-\beta} \phi}\left(T\in  \tau f \left(\frac{\alpha^2 t}{1+\alpha \beta t}\right) \right) = (1+\alpha \beta t)^{1/2}e^{-\frac{\beta \alpha  f^{\alpha,\beta}(t)^2}{2 (1+\alpha \beta t)} }
\P_0^{ \phi}\left(T\in  \tau f \left(\frac{\alpha^2 t}{1+\alpha \beta t}\right) \right),
\end{eqnarray*}
and thus
\begin{eqnarray*}
p^{{f^{\alpha,\beta}}}\left(t\right)dt&=&f^{-2}\left(\frac{\alpha^2 t}{1+\alpha \beta t}\right)(1+\alpha \beta t)^{1/2} e^{-\frac{\beta \alpha  f^{\alpha,\beta}(t)^2}{2 (1+\alpha \beta t)} }
\P_0^{ \phi}\left(T\in  \tau f \left(\frac{\alpha^2 t}{1+\alpha \beta t}\right) \right).
\end{eqnarray*}
Using again Lemma \ref{Proposition-connecting-times} we finally obtain
\begin{eqnarray*}
p^{{f^{\alpha,\beta}}}\left(t\right)&=& \alpha^{2}(1+\alpha \beta t)^{-3/2} e^{-\frac{\beta \alpha f^{\alpha,\beta}(t)^2}{2 (1+\alpha \beta t)}}
p^{{f}}\left(\frac{\alpha^2 t}{1+\alpha \beta t}\right)\\
&=& \alpha^3 (1+\alpha \beta t)^{-5/2} e^{-\frac{\alpha \beta f^{\alpha,\beta}(t)^2}{2(1+\alpha \beta t)} }S^{\alpha,\beta}p^{{f}}\left(t\right)
\end{eqnarray*}
which is the main identity \eqref{eq:fptr}.
\subsection{ The  approach via the Lie group symmetries of the heat equation}
The purpose of this proof is to show how the symmetry groups method may be used to derive our main identity \eqref{eq:fptr}.  We recall that the application of Lie group theory to solve differential equations dates back to the original work of S. Lie. It provides an effective mechanism for computing a wide variety of new solutions of a specific differential equation from known ones. An excellent account of this technique can be found in the monograph of Olver \cite{Olver}. We mention that recently Lescot and Zambrini \cite{LZ} resort to Lie
 group techniques for the study of some diffusions with a view towards  stochastic symplectic geometry.  Before describing the symmetries of the heat equation, that is the one-parameter group of transformations leaving invariant the space  of positive solutions of this equation, denoted throughout by $\mathcal{H}$, we state the following result which relates the boundary crossing problem to the study of the heat equation.  These claims can be found in  Theorem 1.1 and Lemma 1.4  of Lerche \cite{Lerche-86}.
\begin{prop}
Let us assume that the function $f$ is infinitely continuously differentiable on $\R^+$ and define the domain $D^f$ by $D^f = \{(x,t) \in \R\times \R^+; \: x\leq f(t)\}$.  Then there exists a unique (strong) solution to the following boundary value problem
\begin{equation}\label{eq:he}
 \hspace{- 2cm} \mathcal{H}(f): \hspace{1cm}
\begin{cases}
 \frac{\partial h}{\partial t} (x,t)=\frac{1}{2}\frac{\partial^2 h}{\partial x^2} (x,t)& \textrm{ on } D^f,\nonumber \\
 h(f(t),t) = 0 & \textrm{ for all } t>0,\\
 h(.,0) = \delta_0(.) & \textrm{ on } (-\infty,f(0)), \nonumber \\
  \end{cases}
 \end{equation}
 where $\delta_0$ stands for the dirac point mass at $0$. Moreover, $h$ admits the following probabilistic representation
 \begin{equation} \label{eq:pr}
  h(x,t) dx = \P\left(B_t \in dx, t<T^f\right).
  \end{equation}
 Finally, the law of $T^f$ is absolutely continuous with a continuous density given by
  \begin{equation}\label{eq:df}
   p^f(t) =-\frac{1}{2}\frac{\partial h }{\partial x} (x,t) _{| x=f(t)}.\end{equation}
\end{prop}
   \begin{remark}
   We point out that one  may weaken the above assumption that $f$ is infinitely continuously differentiable in order that the solution of the boundary value problem $\mathcal{H}(f)$ admits the probabilistic representation \eqref{eq:pr}. We refer  the interested readers to the monograph of Friedman \cite{Friedman}.
   \end{remark}



  We now turn to the calculation of symmetry groups of the heat equation
\begin{equation}
  \frac{\partial h}{\partial t} (x,t)=\frac{1}{2}\frac{\partial^2 h}{\partial x^2} (x,t).
\end{equation}
 This equation has been intensively studied and its symmetry groups are well known. A detailed description of the procedure for finding its most general one parameter group of transformations   is given in \cite[Section 2.4, p.117]{Olver}. However, therein the heat equation is considered without the factor $\frac{1}{2}$. This obviously affects slightly the symmetry groups but we shall provide the correct expression for the Lie algebra basis and the corresponding  transformation. After some easy but tedious computation, one finds that the Lie algebra of infinitesimal symmetries  of the heat equation (with the factor $\frac{1}{2}$) is spanned by the six vector fields, where $x,t$ are the two independent variables and $h$ is the dependent variable,
 \begin{eqnarray*}
  {\bf{v}}_1 &=& \frac{\partial }{\partial x} , \: {\bf{v}}_2 = \frac{\partial }{\partial t} , \: {\bf{v}}_3 = h\frac{\partial }{\partial h}, \\
   {\bf{v}}_4 &=& x\frac{\partial }{\partial x} + 2 t \frac{\partial }{\partial t}, \: {\bf{v}}_5 = t\frac{\partial }{\partial x} - {\bf{2}} x h\frac{\partial }{\partial h},\\
    {\bf{v}}_6 &=& 4tx\frac{\partial }{\partial x} + 4 t^2 \frac{\partial }{\partial t} -({\bf{2}}x^2+2t)h\frac{\partial }{\partial h},
 \end{eqnarray*}
 and the infinite-dimensional subalgebra ${\bf{v}}_{u}=u(x,t) \frac{\partial }{\partial h}$ where $u$ is an arbitrary solution of the heat equation (we put in bold face the (two) coefficients of the vector fields which have been affected by the factor $\frac{1}{2}$). Note that the 6-dimensional Lie algebra is the semidirect sum of $\mathfrak{sl}(2,\mathbb{R})$ with the Heisenberg-Weyl algebra.  Exponentiating the basis produces the following one-parameter group of transformations leaving invariant $\mathcal{H}$, for any $\varepsilon \in \R$,
\begin{eqnarray*}
h^{(1)}(x,t) &=& \exp(\varepsilon  {\bf{v}}_1)h(x,t)= h(x-\varepsilon,t),\\
h^{(2)}(x,t) &=& \exp(\varepsilon  {\bf{v}}_2)h(x,t)= h(x,t-\varepsilon),\\
h^{(3)}(x,t) &=& \exp(\varepsilon  {\bf{v}}_3)h(x,t)=e^{\varepsilon} h(x,t),\\
h^{(4)}(x,t) &=& \exp(\varepsilon  {\bf{v}}_4)h(x,t)= h(e^{-\varepsilon}x,e^{-2\varepsilon}t),\\
h^{(5)}(x,t) &=& \exp(\varepsilon  {\bf{v}}_5)h(x,t)= e^{-{\bf{4}}\varepsilon x+{\bf{8}}\varepsilon^2 t}h(x-2\varepsilon t,t),\\
h^{(6)}(x,t) &=& \exp(\varepsilon  {\bf{v}}_6)h(x,t)=\frac{1}{\sqrt{1+4\varepsilon t}}e^{-\frac{{\bf{2}} \varepsilon x^2}{1+4 \varepsilon t}} h\left(\frac{ x}{1+4 \varepsilon t},\frac{t}{1+4 \varepsilon t}\right),\\
h^{(u)}(x,t) &=& \exp(\varepsilon  {\bf{v}}_{u})h(x,t)=h(x,t) +\varepsilon u(x,t).
\end{eqnarray*}
 Note that the symmetry groups  associated to these  transformations  provide the explanation of the invariance of the law of the Brownian motion under some specific transformations. More precisely, the groups $h^{(1)}$ and $h^{(2)}$ show the time-and space-invariance of the law of the Brownian motion, the scaling property turns up in the composition of  $h^{(3)}$ and $h^{(4)}$, while $h^{(5)}$ represents the Girsanov transform (Doob's $h$-transform) connecting the law of Brownian motion with  different drifts. The group $h^{(6)}$ is intimately connected to the change of measure  (another Doob's $h$-transform) connecting the law of the Brownian motion with its bridges and thus, from  Pitman and Yor \cite{Pitman-Yor-82a}, one can explain the time-inversion property of the Brownian motion.
One can obviously build up some transformations leaving invariant the space $\mathcal{H}$ by considering any compositions of the above transformations. However, in the boundary crossing context, the most original and relevant combination of the above symmetries appears to be the  two-parameter transformation $h^{(\alpha,\beta)}$ defined, for any $\alpha>0$ and $\beta\in \R$, by
\begin{eqnarray} \label{eq:th}
h^{(\alpha,\beta)}(x,t) &=& \exp(\ln(\alpha)  {\bf{v}}_3) \circ \exp(-\ln(\alpha)  {\bf{v}}_4) \circ \exp\left(\frac{\beta }{4 \alpha}  {\bf{v}}_6\right)  h(x,t) \nonumber\\
&=&\frac{\alpha}{\sqrt{1+\alpha \beta t}}e^{-\frac{\alpha \beta x^2}{2(1+\alpha \beta t)}} h\left(\frac{\alpha x}{1+\alpha \beta  t},\frac{\alpha ^2 t}{1+\alpha \beta t}\right).
\end{eqnarray}
 We are now ready to state the following result which explains how this specific composition of symmetries  affects the domain and the boundary conditions of the boundary value problem $\mathcal{H}(f)$ defined above.
 \begin{prop}
 Let $h$ be the solution of the boundary value problem $\mathcal{H}(f)$. Then, for any $\alpha>0$, $\beta \in \R$, the mapping $h^{(\alpha,\beta)}$ defined in \eqref{eq:th} is the solution to the boundary value problem $\mathcal{H}(f^{\alpha,\beta})$.
\end{prop}
 \begin{proof}
We assume without loss of generality that $\beta>0$. It is plain that if $f$ is infinitely continuously differentiable then so is $f^{\alpha,\beta}$. Then, from the symmetry property of the transformations and the fact that $h$ is a solution to the heat equation on $D^f$, it is easy to check that the function $h^{(\alpha,\beta)}$ is a solution to the heat equation on $D^{f^{\alpha,\beta}}$. Moreover, we have, for all $t>0$,
\begin{eqnarray*}
h^{(\alpha,\beta)}(f^{\alpha,\beta}(t),t)&=&\frac{\alpha}{\sqrt{1+\alpha \beta t}}e^{-\frac{\alpha \beta (f^{\alpha,\beta}(t))^2}{2(1+\alpha \beta t)}} h\left(\frac{\alpha f^{\alpha,\beta}(t)}{1+\alpha \beta  t},\frac{\alpha ^2 t}{1+\alpha \beta t}\right)\\
&=&\frac{\alpha}{\sqrt{1+\alpha \beta t}}e^{-\frac{\alpha \beta (f^{\alpha,\beta}(t))^2}{2(1+\alpha \beta t)}} h\left(f\left(\frac{\alpha ^2 t}{1+\alpha \beta t}\right),\frac{\alpha ^2 t}{1+\alpha \beta t}\right)\\
&=&0
\end{eqnarray*}
  since $h\left(f(t),t\right)=0$ for all $t>0$. On the other hand, observing from the scaling property of the Brownian motion $B$ that, for any $\alpha>0$, $h^{(\alpha,0)}$ is a solution to  the problem $\mathcal{H}(f^{\alpha,0})$ and thus in particular $h^{(\alpha,0)}(.,0)=\delta_0(.)$ on  $ (-\infty,  \frac{f(0)}{\alpha})$. We complete the proof by noting that
\begin{eqnarray*}
h^{(\alpha,\beta)}(x,0)&=&e^{-\frac{\alpha \beta x^2}{2}} h^{(\alpha,0)}\left(\alpha x,0\right)\\
&=&\delta_0(x).
\end{eqnarray*}
 \end{proof}
 It is now an easy exercise to derive our main identity   \eqref{eq:fptr}. Indeed, since
from \eqref{eq:df}, we have  $p^{f^{\alpha,\beta}}(t) =-\frac{1}{2}\frac{\partial h^{(\alpha,\beta)}}{\partial x} (x,t) _{| x=f^{\alpha,\beta}(t)}$. Thus, differentiating \eqref{eq:th} and using the condition $h\left(f(t),t\right)=0$ for all $t>0$, we get that
\begin{eqnarray} \label{eq:gd}
p^{f^{\alpha,\beta}}(t) &=&-\frac{1}{2}\frac{\alpha^2}{(1+\alpha \beta t)^{3/2}}e^{-\frac{\alpha \beta (f^{\alpha,\beta}(t))^2}{2(1+\alpha \beta t)}} \frac{\partial h }{\partial x} \left( x,\frac{\alpha ^2 t}{1+\alpha \beta t}\right)_{| x=f\left(\frac{\alpha ^2 t}{1+\alpha \beta t}\right)}.
\end{eqnarray}
Using successively  the fact that $h$ is solution to the heat equation on $D^f$ and  the condition $h(f(t),t)=0$ for all $t>0$  yields
\begin{eqnarray*}
 \frac{\partial h }{\partial x} \left( x,\frac{\alpha ^2 t}{1+\alpha \beta t}\right)_{| x=f\left(\frac{\alpha ^2 t}{1+\alpha \beta t}\right)} &=& \int_{-\infty}^{f\left(\frac{\alpha ^2 t}{1+\alpha \beta t}\right)} \frac{\partial^2 h}{\partial y^2} \left(y,\frac{\alpha ^2 t}{1+\alpha \beta t}\right)dy\\
&=& 2\int_{-\infty}^{f\left(\frac{\alpha ^2 t}{1+\alpha \beta t}\right)} \frac{\partial h}{\partial t} \left(y,\frac{\alpha ^2 t}{1+\alpha \beta t}\right)dy\\
& =& 2 \frac{(1+\alpha \beta t)^2}{\alpha ^2 }\frac{d}{dt}
\int_{-\infty}^{f\left(\frac{\alpha ^2 t}{1+\alpha \beta t}\right)}  h\left(y,\frac{\alpha ^2 t}{1+\alpha \beta t}\right)dy \\
& =& 2 \frac{(1+\alpha \beta t)^2}{\alpha ^2 }\frac{d}{dt} \P\left(T^{f}>\frac{\alpha ^2 t}{1+\alpha \beta t}\right)
 \\
& =&- 2 p^{{f}}\left(\frac{\alpha ^2 t}{1+\alpha \beta t}\right)
\end{eqnarray*}
which combined with the identity \eqref{eq:gd} gives \eqref{eq:fptr}.

 We point out that by following a similar line of reasoning, we can apply the group transformation $h^{(5)}$ to the boundary crossing problem.  That results in adding a linear trend to the considered curve.  The latter transformation can also be composed in a similar way with the transformation $h^{(\alpha,\beta)}$.   The symmetry group approach has the great advantage to explain that beside the transformations we just discussed, there does not exist other simple and attainable identities relating first passage time distributions  to some parametric family of curves.

\bibliographystyle{amsplain}

\begin{thebibliography}{10}


\bibitem{Alili-patie-JTP-09}
L.~Alili and P.~Patie.
\newblock {Boundary-crossing identities for diffusions having the time-inversion property}.
\newblock {\em J. Theoret. Probab.}, 23(1):65--84, 2009.

\bibitem{Bachelier}
L.~Bachelier
\newblock {\em Th\'eorie de la sp\'eculation},
    \newblock {Les Grands Classiques Gauthier-Villars},
                  Reprint of the 1900 original,
\'Editions Jacques Gabay, 1995.

\bibitem{Breiman-67}
L.~Breiman.
\newblock First exit times from a square root boundary.
\newblock In {\em Proc. Fifth Berkeley Sympos. Math. Statist. and Probability
  (Berkeley, Calif., 1965/66), Vol. II: Contributions to Probability Theory,
  Part 2}, pages 9--16.  1967.

\bibitem{Doob-42}
J.L. Doob.
\newblock The {Brownian} movement and stochastic equations.
\newblock {\em Ann. of Math.}, 43(2):351--369, 1942.


\bibitem{Friedman}
A.~Friedman.
\newblock { \em Partial differential equations of parabolic type},
\newblock {Prentice-Hall Inc., Englewood Cliffs, N.J.}, 1964.



\bibitem{Groeneboom}
P.~Groeneboom.
\newblock {Brownian motion with a parabolic drift and {A}iry functions},
\newblock {\em Probab. Theory Related Fields}, 81(1):79--109, 1989.



\bibitem{Lerche-86}
H.R. Lerche.
\newblock Boundary crossing of {B}rownian motion: Its relation to the law of
  the iterated logarithm and to sequential analysis.
\newblock {\em Lecture Notes in Statistics}, 40, 1986.

\bibitem{LZ}
P. Lescot and J.-C. Zambrini.
\newblock{ Probabilistic deformation of contact geometry,
diffusion  processes and their quadratures},
\newblock{\em Seminar on {S}tochastic {A}nalysis, {R}andom {F}ields and
              {A}pplications {V}, Progr. Probab.}, 59:203--226,
Birkh\"auser, Basel, 2008.

\bibitem{Olver}
P.J. Olver.
\newblock{ \em Applications of {L}ie groups to differential equations},
 Vol. 107, $2^{nd}$ edition, Springer-Verlag, New York, 1993.



\bibitem{Pitman-Yor-82a}
J.~Pitman and M.~Yor.
\newblock {{A} decomposition of {B}essel bridges}.
\newblock {\em Z. Wahrsch. Verw. Gebiete}, 59:425--457, 1982.

\bibitem{Revuz-Yor-99}
D.~Revuz and M.~Yor.
\newblock {\em Continuous Martingales and Brownian Motion}, Vol. 293.
\newblock Springer-Verlag, Berlin-Heidelberg, $3^{rd}$ edition, 1999.


\bibitem{Sal}
P.~Salminen.
\newblock{On the First Hitting Time and the Last Exit Time for a {B}rownian
	Motion to/from a Moving Boundary.},
 \newblock {Adv. in Appl. Probab.},
Vol. 20(2):411--426, 1988.


\bibitem{Strassen-67}
V.~Strassen.
\newblock {{A}lmost sure behavior of sums of independent random variables and
  martingales}.
\newblock In {\em Proc. Fifth Berkeley Sympos. Math. Statist. and Probability (Berkeley, Calif., 1965/66) Vol. II: Contributions to Probability Theory, Part 1,} pages 315--343, 1967.

\end{thebibliography}

\end{document}